\documentclass{amsart}
\usepackage{amsfonts}

\newtheorem{theorem}{Theorem}
\theoremstyle{plain}

\newtheorem{corollary}{Corollary}

\newtheorem{lemma}{Lemma}

\numberwithin{equation}{section}

\begin{document}
\title{Fibonacci Designs}
\author{Harold N. Ward}
\address{Department of Mathematics\\
University of Virginia\\
Charlottesville, VA 22904\\
USA}
\email{hnw@virginia.edu}

\begin{abstract}
A Metis design is one for which $v=r+k+1$. This paper deals with Metis
designs that are quasi-residual. The parameters of such designs and the
corresponding symmetric designs can be expressed by Fibonacci numbers.
Although the question of existence seems intractable because of the size of
the designs, the nonexistence of corresponding difference sets can be dealt
with in a substantive way.

We also recall some inequalities for the number of fixed points of an
automorphism of a symmetric design and suggest possible connections to the
designs that would be the symmetric extensions of Metis designs.
\end{abstract}

\maketitle

\section{Quasi-residual Metis designs}

In the paper \cite{MMW} by McDonough, Mavron, and the author, a method of
amalgamating nets and designs was presented that led to quasi-symmetric
designs similar to those discovered by Bracken, McGuire, and the author \cite%
{BMW}. At various stages of the construction, restrictions on the designs
involved needed to be imposed in order to make the final amalgamation have
desired regularity properties. One particular type of design was a
generalization of Hadamard designs, and we named them \emph{Metis designs},
in honor of Hadamard's ancestral home, Metz. They are block designs whose
standard parameter set $(v,b,r,k,\lambda )$ satisfies the additional
relation $v=r+k+1$. Symmetric Metis designs are indeed Hadamard designs. The
family of Metis designs $\mathcal{M}$ has the following property: regard the
parameter set for any design as a point in $\mathbb{R}^{5}$ on the variety $%
\mathcal{D}$ defined by the two standard design relations $vr=bk$ and $%
r(k-1)=\lambda (v-1)$. Then if a design belongs to $\mathcal{M}$, there is a
line in $\mathcal{D}$ through the corresponding point such that all the
points on that line belong to $\mathcal{M}$. There are other common families
of designs with such a linear property. The nature of lines in $\mathcal{D}$
has been explored in the somewhat speculative preprint \cite{WD}, and
Section \ref{WD} describes some of the results.

It seems a natural question to ask for Metis designs that are also
quasi-residual. The parameters would satisfy the two equations%
\begin{eqnarray*}
v &=&r+k+1 \\
r &=&k+\lambda ,
\end{eqnarray*}%
along with the two standard equations. Solving the four equations in terms
of $r$, we get $k=(\sqrt{(5r+4)r}-r)/2$. Let $d=\gcd (5r+4,r)$. Then $d|4$,
and $(5r+4)/d$ and $r/d$ must separately be squares. Reading modulo 4, we
find that $d=2$ will not work, and so $d=1$ or 4. Thus $5r+4$ and $r$
themselves must be squares. Put $5r+4=x^{2}$ and $r=y^{2}$, with $x,y>0$.
Let $F_{t}$ and $L_{t}$ be the $t$-th Fibonacci and Lucas numbers,
respectively, starting with $F_{0}=0$ and $L_{0}=2$ (see the books by Koshy 
\cite{K} or Vajda \cite{V}, for example, which were sources for several
other citations). One has

\begin{lemma}
\cite[Lemma 2]{LJ} For some $t\geq 1$, $x=L_{2t}$ and $y=F_{2t}$.
\end{lemma}

\noindent We substitute these values for $x$ and $y$ in the design
parameters, make use of the relation $L_{2t}=F_{2t+1}+F_{2t-1}$ \cite[%
Formula (6)]{V}, and invoke the basic recurrence $F_{t+2}=F_{t+1}+F_{t}$, to
obtain%
\begin{equation*}
v=F_{m}F_{m-1}+1,\,b=F_{m+1}F_{m-1},\,r=F_{m-1}^{2},\,k=F_{m-1}F_{m-2},\,%
\lambda =F_{m-1}F_{m-3,}
\end{equation*}%
where $m$ is odd. The parameters $(v^{\prime },k^{\prime },\lambda ^{\prime
})$ for a symmetric design having the Metis design as its block residual are 
$v^{\prime }=b^{\prime }=b+1$, $r^{\prime }=k^{\prime }=r$, and $\lambda
^{\prime }=\lambda $. By the relation $F_{m+1}F_{m-1}+1=F_{m}^{2}$ \cite[%
Formula (29)]{V},%
\begin{equation*}
v^{\prime }=F_{m}^{2},\,k^{\prime }=F_{m-1}^{2},\,\lambda ^{\prime
}=F_{m-1}F_{m-3}.
\end{equation*}%
The order of this design is $n^{\prime }=k^{\prime }-\lambda ^{\prime
}=F_{m-1}^{2}-F_{m-1}F_{m-3}=F_{m-1}F_{m-2}$. Prompted by the appearance of
the Fibonacci numbers, we call a symmetric design $\mathcal{F}_{m}$ with
parameters $(v^{\prime },k^{\prime },\lambda ^{\prime
})=(F_{m}^{2},F_{m-1}^{2},F_{m-1}F_{m-3})$, $m$ odd, a \emph{Fibonacci design%
}. As we shall concentrate on these symmetric designs, we drop the dashes
for clarity.

\section{Existence of Fibonacci designs}

The design $\mathcal{F}_{3}$ is the trivial $(4,1,0)$ design whose blocks
are the singleton sets, so we may assume that $m>3$ ($m$ always odd) from
now on. There are 78 inequivalent $(25,9,3)$ designs $\mathcal{F}_{5}$, a
classification due to Denniston \cite{D}. The first such design was
presented by Bhattacharya \cite{B}, and an $\mathcal{F}_{5}$ appears as one
of the sequence of designs constructed by Mitchell \cite{M}.

The initial step is to check whether the parameter set passes the
Bruck-Ryser-Chowla criterion (see, for example, \cite[Theorem 10.3.1]{H}).
The variety parameter $v=F_{m}^{2}$ is even just when $m$ is also divisible
by 3, that is, when $m\equiv 3\ (\bmod\ 6)$. This is a consequence of the
periodicity of the $F_{t}$ modulo any given integer; see \cite{W} for
generalities. In this case, the demand is that the order $F_{m-1}F_{m-2}$ be
a square. As consecutive Fibonacci numbers are relatively prime \cite[p. 73]%
{V}, both $F_{m-1}$ and $F_{m-2}$ would have to be squares. But for $t>0$, $%
F_{t}$ is a square just at $t=1,2,12$ \cite{C}. Thus the only existing
design with $v$ even is $\mathcal{F}_{3}$.

Now suppose that $v$ is odd, so that $m\equiv \pm 1\ (\bmod\ 6)$. Then $%
v=F_{m}^{2}\equiv 1\ (\bmod\ 4)$, and the Diophantine equation of the
Bruck-Ryser-Chowla theorem is%
\begin{equation*}
Z^{2}=nX^{2}+\lambda Y^{2}=F_{m-1}F_{m-2}X^{2}+F_{m-1}F_{m-3}Y^{2}.
\end{equation*}%
Here there is an easy solution: $X=Y=1$ and then $Z=F_{m-1}$!

\subsection{Difference set development\label{subsecdiffset}}

With that obstacle to existence removed for $m\equiv \pm 1\ (\bmod\ 6)$, how
might one construct $\mathcal{F}_{m}$? A natural thing to try is to produce $%
\mathcal{F}_{m}$ as the development of a difference set (\textquotedblleft
development\textquotedblright\ for short). In what follows, $z^{\ast }$ is
the square-free part of the integer $z$.

\begin{theorem}
\label{Thmoddorder}Suppose that $\mathcal{F}_{m}$ is the development of a
difference set in a group (Abelian or not), with $m\equiv \pm 1\ (\bmod\ 6)$%
. Let $p$ be a prime dividing $F_{m}$. Then any prime dividing either $%
F_{m-1}^{\ast }$ or $F_{m-2}^{\ast }$ has odd order modulo $p$. Moreover, $%
p\equiv 1\ (\bmod\ 8)$.
\end{theorem}

\begin{proof}
The theorem is a consequence of \cite[Theorem 4.4]{L}, which can be
paraphrased in the present situation to say that for no prime $p$ dividing $%
v $ is there a prime dividing $n^{\ast }$ that has even order modulo $p$. As 
$v=F_{m}^{2}$, such $p$ are the prime divisors of $F_{m}$. Because $F_{m-1}$
and $F_{m-2}$ are relatively prime, $n^{\ast }=F_{m-1}^{\ast }F_{m-2}^{\ast
} $, and the primes dividing $n^{\ast }$ are those dividing either $%
F_{m-1}^{\ast }$ or $F_{m-2}^{\ast }$. Thus such primes must have odd order
modulo $p$. Now $F_{m-1}^{2}+1=F_{m}F_{m-2}$, by \cite[Formula (29)]{V}
again. Put $F_{m-1}=a^{2}F_{m-1}^{\ast }$. Then $a^{4}(F_{m-1}^{\ast
})^{2}\equiv -1\ (\bmod\ p)$. Since all the primes dividing $F_{m-1}^{\ast }$
must have odd order modulo $p$, $F_{m-1}^{\ast }$ itself has odd order
modulo $p$. Thus if $u$ is the odd factor of $p-1$, then $(a^{u})^{4}\equiv
-1\ (\bmod\ p)$. Hence the multiplicative group of $\mathbb{Z}_{p}$ has
order divisible by 8, making $p\equiv 1\ (\bmod\ 8)$.
\end{proof}

\begin{corollary}
\label{Coroddorder}If $\mathcal{F}_{m}$ is a development for some $m$ with $%
m\equiv \pm 1\ (\bmod\ 6)$, then in fact $m\equiv \pm 1\ (\bmod\ 12)$.
\end{corollary}

\begin{proof}
By the observation on square Fibonacci numbers, $n^{\ast }\neq 1$. From the
theorem, all prime divisors $p$ of $F_{m}$ have $p\equiv 1\ (\bmod\ 8)$, so
that $F_{m}\equiv 1\ (\bmod\ 8)$. The period of congruences modulo 8 of the $%
F_{t}$ is 12, and the residue sequence is%
\begin{equation*}
0,1,1,2,3,5,0,5,5,2,7,1,0,1,1,\ldots .
\end{equation*}%
Thus $m\equiv \pm 1\ (\bmod\ 12)$.
\end{proof}

\noindent For instance, suppose, indeed, that $m\equiv \pm 1\ (\bmod\ 12)$.
When 5 divides $m$, 5 also divides $F_{m}$. So if $m\equiv \pm 25\ (\bmod\
60)$, then $m\equiv \pm 1\ (\bmod\ 12)$ all right; but $F_{m}$ has the prime
divisor $5\not\equiv 1\ (\bmod\ 8)$, and no $\mathcal{F}_{m}$ can be a
development.

Even if all primes $p$ dividing $F_{m}$ do have $p\equiv 1\ (\bmod\ 8)$, 
\emph{every} prime dividing $F_{m-1}^{\ast }$ or $F_{m-2}^{\ast }$ must have
odd order modulo such $p$ if $\mathcal{F}_{m}$ is a development. Here is a
consequence:

\begin{theorem}
Let $m\equiv 1\pm 36\ (\bmod\ 216)$. Then no $\mathcal{F}_{m}$ is a
development.
\end{theorem}

\begin{proof}
Here $m-1=36(6h\pm 1)$ for some $h$. By Lucas' law of repetition, as given
more sharply in \cite[Theorem X]{Ca} (referenced in \cite[p. 13]{R}), the
exact power of 3 dividing $F_{m-1}$ is $3^{3}$, the exact power dividing $%
F_{36}$. Thus $3|F_{m-1}^{\ast }$. In addition, $m\equiv 5\ (\bmod\ 8)$
gives $F_{m}\equiv 2\ (\bmod\ 3)$, from the sequence $0,1,1,2,0,2,2,1,0,%
\ldots $ of remainders of $F_{t}$ modulo 3. Thus for some prime $p$ dividing 
$F_{m}$, $p\equiv 2\ (\bmod\ 3)$, so that $\left(\frac{p}{3}\right)=-1$. If
an $\mathcal{F}_{m}$ is a development, Theorem \ref{Thmoddorder} implies
that $p\equiv 1\ (\bmod\ 8)$. Then $\left(\frac{3}{p}\right)=-1$, from
quadratic reciprocity. But that means 3 cannot have odd order modulo $p$.
\end{proof}

There is a similar approach for the $m\equiv -1\ (\bmod\ 12)$, for which 2
is a divisor of $F_{m-2}^{\ast }$. If 2 has even order modulo $F_{m}$, then
for some prime $p$ dividing $F_{m}$, 2 will also have even order modulo $p$,
and a development model for $\mathcal{F}_{m}$ will be ruled out. The
computation of the order is rather formidable, and the $m$ for which this
works do not seem to follow an obvious pattern. The first few are $m=35$,
47, 59, 71, 95, 107, 119, 143, 155, 167, 191, and 203. (In point of fact, $%
m=35$, 95, 119, 143, 155, and 203 are also ruled out for $\mathcal{F}_{m}$
developments by prime factors of $F_{m}$ not congruent to 1 modulo 8.)

We have $F_{m-1}^{2}\equiv -1\ (\bmod\ F_{m})$, and in addition, $%
F_{m-2}^{2}\equiv -1\ (\bmod\ F_{m})$, since $F_{m-2}\equiv -F_{m-1}\ (\bmod%
\ F_{m})$. Thus if an $\mathcal{F}_{m}$ is a development and either of $%
F_{m-1}$ or $F_{m-2}$ is square-free, one of these congruences is
contradicted by the fact that $F_{m-1}^{\ast }$ and $F_{m-2}^{\ast }$ must
have odd order modulo any prime divisor of $F_{m}$. Such a contradiction is
in fact what happens for all $m\equiv \pm 1\ (\bmod\ 12)$ with $m<1000$,
except for $m=277$, 457, 577, and 877. (The web page by B. Kelly \cite{Ke}
contains factorizations of the first thousand Fibonacci numbers, along with
tables that can be used for the first \emph{ten} thousand.) For the first
three of these $m$, the possibility that $\mathcal{F}_{m}$ exists as a
development is ruled out by a factor of $F_{m}$ not congruent to 1 modulo 8.
For $m=877$, the prime 1753 divides $F_{877}$, and $F_{875}$ is exactly
divisible by $5^{3}$. Thus 5 is in the square-free part of $F_{875}$. As the
order of 5 modulo 1753 is 584, the odd-order requirement does not hold. In
short, no $\mathcal{F}_{m}$ for $m\equiv \pm 1\ (\bmod\ 12)$ with $m<1000$
can be a development. That $\mathcal{F}_{5}$ is not one is recorded in \cite[%
Table 4-1]{L}; and that a $(169,64,24)$ design $\mathcal{F}_{7}$ cannot be a
development is presented in \cite{Ko}. Whether such a design exists at all
seems to be unknown.

All of these results suggest the conjecture that apart from the trivial
design $\mathcal{F}_{3}$, no Fibonacci design is a development of a
difference set--and this seems to depend on mysterious properties of
Fibonacci numbers.

\section{Bounds}

A standard approach to constructing symmetric designs is to prescribe a
group of automorphisms, thereby limiting the choices needed to be explored
in the construction. The investigation of the possible orbit structure of an
automorphism of specified order is an important starting point. The fact
that the cycle structures of the permutations induced on the points and
blocks by an automorphism are the same \cite{P} is a key ingredient.

There are various bounds on the number of fixed points of an automorphism of
a symmetric design; some are summarized and proved in \cite[Section 3.1]{L}.
We shall present one due to Bowler \cite[Lemma 2.5 (i)]{Bo} with essentially
his proof. This bound was later generalized by Feit \cite{F}.

\begin{theorem}
The three-block bound: Let $\sigma $ be an automorphism of a symmetric $%
(v,k,\lambda )$ design $\mathcal{D}$ of order $n=k-\lambda $, and let $%
\sigma $ have $f$ fixed points. Then if the order of $\sigma $ is at least
3, $f\leq v-3n$.
\end{theorem}

\begin{proof}
Take $l$ to be the length of a longest cycle in the action of $\sigma $ on
the blocks of $\mathcal{D}$, and let $B_{1}$, $B_{2}$, and $B_{3}$ be three
distinct blocks in that cycle. With $\left\vert X\right\vert $ denoting the
size of a set $X$, we have $\left\vert B_{1}\cup B_{2}\cup B_{3}\right\vert
=3k-3\lambda +\left\vert B_{1}\cap B_{2}\cap B_{3}\right\vert $. Since any
fixed point of $\sigma $ in one of the $B_{i}$ is in all three, there are no
fixed points in $B_{1}\cup B_{2}\cup B_{3}-B_{1}\cap B_{2}\cap B_{3}$. Thus%
\begin{equation*}
f\leq v-\left\vert B_{1}\cup B_{2}\cup B_{3}-B_{1}\cap B_{2}\cap
B_{3}\right\vert =v-3(k-\lambda )=v-3n.
\end{equation*}
\end{proof}

Suppose that equality holds in the three-block bound: $f=v-3n$. Continue
with the notation in the proof, and let $B$ be a representative block in the
chosen longest cycle of $\sigma $. Then the complement of $B_{1}\cup
B_{2}\cup B_{3}-B_{1}\cap B_{2}\cap B_{3}$ is the set $F$ of fixed points of 
$\sigma $; and $B_{1}\cap B_{2}\cap B_{3}$ must be the set $F_{0}$ of fixed
points in $B$, which is the set of fixed points in each $B_{i}$. The set $%
W=B_{1}\cup B_{2}\cup B_{3}-B_{1}\cap B_{2}\cap B_{3}$ is the union of all
the orbits of $\sigma $ other than the fixed points. Each such orbit must
accordingly have a point in $B$. Moreover, $B$ cannot contain a complete
such orbit, for then it would be in all the $B_{i}$ and so in $B_{1}\cap
B_{2}\cap B_{3}$ and yet not consist of fixed points. In particular, $\sigma 
$ cannot have a cycle with length $t>1$ and relatively prime to $l$. Because 
$W=B_{1}\cup B_{2}\cup B_{3}-B_{1}\cap B_{2}\cap B_{3}$, $B_{3}\supseteq
F_{0}\cup (W-(B_{1}\cup B_{2}))$. With $\left\vert F_{0}\right\vert =f_{0}$,
then since $\left\vert W\right\vert =3n$ and $F_{0}$ is disjoint from $W$, $%
\left\vert F_{0}\cup (W-(B_{1}\cup B_{2}))\right\vert =2f_{0}+k-2\lambda $.
If $l\geq 4$, the same holds for a fourth block $B_{4}$ in the orbit, in
place of $B_{3}$. Then $B_{3}\cap B_{4}\supseteq F_{0}\cup (W-(B_{1}\cup
B_{2}))$, so that $\lambda \geq 2f_{0}+k-2\lambda $, and $f_{0}\leq k-3n/2$.

Assume now that $l\geq 5$. If $B$ contains a point $x$ in a 2-cycle of $%
\sigma $, then $x\in B\cap B^{\sigma ^{2}}\cap B^{\sigma ^{4}}$ but $x\notin
F_{0}$. Similarly, if $x$ and $y$ are together in a 3- or 4- cycle of $%
\sigma $ and in $B$, then with proper choice of the $B_{i}$, one would find $%
x\in B_{1}\cap B_{2}\cap B_{3}$. So 2-cycles of $\sigma $ on points do not
meet $B$ at all, and 3- or 4- cycles meet $B$ at most once. Likewise, if $%
x,y,z\in B$ are in a common $t$-cycle of $\sigma $ with $t\geq 5$, then as $%
t\leq l$, $x\in B_{1}\cap B_{2}\cap B_{3}$ for certain $B_{i}$. Thus a $t$%
-cycle of $\sigma $ with $t\geq 5$ meets $B$ at most twice.

Now let $c_{t}$ be the number of $t$-cycles of $\sigma $ for $t>1$. Then the
preceding considerations give 
\begin{equation}
k\leq f_{0}+c_{3}+c_{4}+2\sum_{t\geq 5}c_{t}.  \label{ineq1}
\end{equation}%
Since $v=f+\sum_{t\geq 2}tc_{t}$, we have 
\begin{equation}
3n=v-f=\sum_{t\geq 2}tc_{t}\geq 5/2(c_{3}+c_{4}+2\sum_{t\geq 5}c_{t}).
\label{ineq2}
\end{equation}%
Therefore $f_{0}\geq k-6n/5$. But $f_{0}\leq k-3n/2$, and that is a
contradiction. Thus:

\begin{corollary}
If equality holds in the three-block inequality, then the length $l$ of the
longest cycle in $\sigma $ is at most 4. Moreover, the order of $\sigma $ is 
$l$.
\end{corollary}

\noindent The order statement follows from the comment that cycle lengths
greater than 1 must be relatively prime to $l$.

It is not hard to find designs with automorphisms of order 3 or 4 meeting
the three-block bound. For example, let $Q$ be a normalized Hadamard matrix
of order 4 (see any of \cite{H, IS, L}, for instance). It has an
automorphism of order 3 fixing the first row and first column. If $H$ is an
Hadamard matrix of order $h$, then the Kronecker product $Q\otimes H$
inherits an automorphism of order 3 fixing the first $h$ rows and columns.
The corresponding Hadamard design is a $(4h-1,2h-1,h-1)$ symmetric design of
order $h$ with an automorphism of order 3 having $h-1=(4h-1)-3h$ fixed
points. Similarly, a normalized Hadamard matrix $E$ of order 8 has an
automorphism of order 4 fixing the first two rows and columns, switching the
third and fourth, and cycling the last four (on proper ordering; this may be
seen from the construction of $E$ as a Sylvester matrix). Then $E\otimes H$
leads to an $(8h-1,4h-1,2h-1)$ Hadamard design of order $2h$ having an
automorphism of order 4 with both 2-cycles and 4-cycles and $%
2h-1=(8h-1)-3\times (2h)$ fixed points.

Our interest here is that appropriate Fibonacci designs may be candidates
for designs with automorphisms meeting the three-block bound. Unfortunately,
there is not much evidence! The design $\mathcal{F}_{3}$ has an automorphism
of order 3 with $1=4-3\times 1$ fixed points. More interestingly, several $%
\mathcal{F}_{5}$ designs can be constructed to have an automorphism of order
3 with $7=25-3\times 6$ fixed points. For example, in the notation of \cite%
{B}, Bhattacharya's design has the automorphism%
\begin{equation*}
(X_{1}Y_{1}Z_{1})(X_{2}Y_{2}Z_{2})(X_{3}Y_{3}Z_{3})(X_{4}Y_{4}Z_{4})(V_{1}V_{2}V_{3})(W_{1}W_{3}W_{2})
\end{equation*}%
on points.

So can such an $\mathcal{F}_{7}$, that is, a $(169,64,24)$ design, be
created? It is tantalizing to observe that the parameters of the $(41,16,6)$
design constructed by Bridges, Hall, and Hayden \cite{BHH} (which, in fact,
has an automorphism of order 3 with $11=41-3\times 10$ fixed points) and $%
(169,64,24)$ are given by $t=2$ and $t=3$ in the sequence%
\begin{equation*}
v=\frac{2q(q^{t}-1)}{q-1}+1,\ k=q^{t},\ \lambda =\frac{1}{2}q^{t-1}(q-1)
\end{equation*}%
for $q=4$. This same parameter sequence, but with $q$ a power of an \emph{odd%
} prime, corresponds to a family of designs discovered by A. E. Brouwer that
is presented in \cite[Section 11.8]{IS}. Incidentally, this parameter set
does not always pass the Bruck-Ryser-Chowla test. One has $n=q^{t-1}(q+1)/2$
and $v\equiv 1\,(\bmod\ 4)$, so the Diophantine equation is%
\begin{equation*}
2Z^{2}=q^{t-1}((q+1)X^{2}+(q-1)Y^{2}).
\end{equation*}%
If $q$ is an even power of 2 or $t$ is even, this has the solution $%
X=1,Y=1,Z=\sqrt{q^{t}}$. If neither condition holds, $q^{t-1}$ is a square
that can be absorbed into $X^{2}$ and $Y^{2}$. The solvability criterion of
Legendre here comes down to the requirement that $-1$ be a square modulo
each prime divisor $p$ of $(q+1)^{\ast }$ (see \cite[pp. 45--46]{L}, among
others); that is, that $p\equiv 1\,(\bmod\ 4)$. For odd powers $q$ of 2, it
is a pleasant exercise to show that only $q=8$ passes: $2\times
12^{2}=9\times 5^{2}+7\times 3^{2}$.

\section{Some results from the paper \protect\cite{WD}\label{WD}}

The affine variety $\mathcal{D}$ in $\mathbb{R}^{5}$ is defined by the two
design equations $vr=bk$ and $r(k-1)=\lambda (v-1)$. A point $\Delta
=(v,b,r,k,\lambda )$ on $\mathcal{D}$ will be referred to as a
\textquotedblleft design,\textquotedblright\ even though it may not
correspond to a genuine block design. Let $Q(\Delta )=Q=r^{2}-b\lambda $.
Then the subvariety $\mathcal{Q}$ of $\mathcal{D}$ defined by $Q=0$ is the
union of four planes:%
\begin{equation*}
\begin{tabular}{ll}
$\Pi _{0}:$ & $b=0,r=0,\lambda =0$ \\ 
$\Pi _{1}:$ & $b=0,r=0,v=1$ \\ 
$\Pi _{2}:$ & $r=0,\lambda =0,k=0$ \\ 
$\Pi _{3}:$ & $b=r=\lambda ,v=k.$%
\end{tabular}%
\end{equation*}%
There are three more planes in $\mathcal{D}$:%
\begin{equation*}
\begin{tabular}{ll}
$\Pi _{4}$ & $v=0,k=0,r=\lambda $ \\ 
$\Pi _{5}$ & $v=1,k=1,b=r$ \\ 
$\Pi _{6}$ & $v=1,k=0,r=0.$%
\end{tabular}%
\end{equation*}%
It seems natural to refer to the designs on these seven planes as degenerate.

The main result of the paper is that on each nondegenerate design $\Delta
_{0}\in \mathcal{D}$, there are exactly four lines that lie in $\mathcal{D}$%
. If $\Delta _{0}=(v_{0},b_{0},r_{0},k_{0},\lambda _{0})$, then one of these
lines is the \emph{replication line}%
\begin{equation*}
v=v_{0},\ b=tb_{0},\ r=tr_{0},\ k=k_{0},\ \lambda =t\lambda _{0},
\end{equation*}%
parameterized by $t$. The other three lines are specified by two parameters, 
$f$ and $p$. For two of the lines, called the $F_{0}$- and $F_{1}$-lines,
the parameter $f$ is the \emph{Fisher factor} of any design on the line: $%
b=fv$. For a genuine design that has $r\neq \lambda $, Fisher's inequality
is that $f\geq 1$ \cite[p. 129]{H}. The fourth line is designated as the $P$%
-line on $\Delta _{0}$; the second parameter $p$ is a bit obscure.

These lines meet certain of the planes $\Pi _{i}$. Suppose that $\Delta _{0}$
is a nondegenerate design and that $\Lambda $, one of the $F_{i}$- or $P$%
-lines on $\Delta _{0}$, meets a plane $\Pi _{j}$. Then if a linear relation%
\begin{equation*}
Vv+Bb+Rr+Kk+L\lambda =A
\end{equation*}%
holds on $\Pi _{j}$ and at $\Delta _{0}$, the relation will hold for all $%
\Delta $ on $\Lambda $. For example, the Metis design relation $v=r+k+1$
holds on the plane $\Pi _{6}$, and all the $F_{1}$-lines meet $\Pi _{6}$. It
follows that the designs on the $F_{1}$-line through a Metis design are all
Metis designs. The paper contains other examples of lines in $\mathcal{D}$
corresponding to families of designs characterized by such linear relations.

\end{document}